\documentclass{amsart}
\usepackage[dvipdfmx]{graphicx}
\usepackage[dvipdfmx]{color}
\usepackage{xcolor}
\usepackage{verbatim}
\usepackage{amsmath}
\usepackage{amssymb, mathrsfs}
\usepackage{amsbsy}
\usepackage{amscd}
\usepackage{amsthm}
\usepackage{bm}

\newcommand{\N}{\mathbb{N}}
\newcommand{\R}{\mathbb{R}}
\newcommand{\Z}{\mathbb{Z}}

\newcommand{\set}[1]{\left\{#1\right\}}

\renewcommand{\d}{{\rm{d}}}

\theoremstyle{break}

\newtheorem{cor}{Corollary}
\newtheorem{thm}{Theorem}
\newtheorem{lem}{Lemma}
\newtheorem{Def}{Definition}
\theoremstyle{definition}
\newtheorem{remark}{Remark}
\usepackage{color}
\definecolor{lilas}{RGB}{182, 102, 210}

\newcommand{\nn}{\nonumber}
\numberwithin{equation}{section}
\begin{document}
\title{Concentration results for directed polymer with unbounded jumps}
\date{\today}
\author{Shuta Nakajima} 
\address[Shuta Nakajima]
{Research Institute in Mathematical Sciences, 
Kyoto University, Kyoto, Japan}
\email{njima@kurims.kyoto-u.ac.jp}

\keywords{directed polymer, random environment, first passage percolation, ground states, zero temperature.}
\subjclass[2010]{Primary 60K37; secondary 60K35; 82A51; 82D30}

\begin{abstract}
We study the free energy and its relevant quantity for the directed polymer in random environment. The polymer is allowed to make unbounded jumps and the environment is given by the Bernoulli variables. We first establish the concentration of the ground state energy of polymer at zero temperature. Secondly, we also prove the same property of the free energy at finite temperature. In the proof, we use the fact that the maximum jump of any polymer nearly minimizing energy is not too large with high probability. This is an interesting property itself from the first passage percolation viewpoint.
\end{abstract}

\maketitle

\section{Introduction}
We discuss models of directed polymer which have unbounded jumps introduced in \cite{CFNY}. See~\cite{CFNY} for the background and related works. The previous paper shows the following three results: (i) The continuity of the free energy with respect to inverse temperature and the appearance probability of obstacles; (ii) The asymptotic of the free energy as the appearance probability of obstacles goes to 1; (iii) The continuity of the time constant of First Passage Percolation (FPP) related to the model. In this paper, we show concentration bounds for the ground state energy, which is nothing but the passage time, and the free energy. As applications, we derive the so-called ``rate of convergence'' results and, based on them, we are able to remove the restriction left in \cite{CFNY} on a parameter in (i) and also give an alternative proof of (iii).

In fact, a concentration bound for the lower tail for FPP is shown in \cite{CFNY}, Proposition 3.1. However the upper tail is significantly more difficult as we need to control the maximum jump of the optimal path, see Section~\ref{KEY} below. In this paper, we succeed in controlling jumps not only of the optimal path but also of low energy paths in the directed polymer model. As a result, we can show the concentration for the passage time and the free energy. 

\subsection{Setting of models}
\hspace{2mm}\\

Let $(\{X_n\}_{n\in\N},P)$ be the random walk on $\Z^d$ starting from $0$ and with the transition probability
\begin{equation} \nn
 P(X_{n+1}=z|X_n={y})=f(|{y-z}|_1),
\end{equation}
where $|x|_1:=\sum_{i=1}^d|x_i|$ for $x\in\Z^d$ and $f:\N\cup\{0\}\to \R$ is a function of the form
\begin{equation}
 f(k)=c_1\exp \{-c_2k^\alpha\},   \label{f(k)}
\end{equation}
where $\alpha,c_2>0$ and $c_1$ is a positive constant determined as to be $\sum_{y\in\Z^d}f(|{y}|_1)=1$ (for the choice of $f(k)$, see Remark \ref{choice} below).
The random environment is modelled by independent and identically
distributed Bernoulli random variables
$(\{\eta(j,x)\}_{(j,x)\in\N\times\Z^d}, Q)$
with parameter $p$;
\begin{equation} \nn
 Q(\eta(0,0)=1)=p\in(0,1).
\end{equation} 
We introduce the Hamiltonian 
\begin{equation} \nn
 H_n^\eta(X)=\sum_{j=1}^n\eta(j,X_j),
\end{equation}
and define the partition functions by 
\begin{equation} \nn
 Z_n^{\eta,\beta}=P[\exp\{\beta H_n^\eta\}] \textrm{ for $\beta\in\R$ and } 
 Z_n^{\eta,-\infty}=P(H_n^\eta=0),
\end{equation}
where $P[\cdot]$ denotes the expectation with respect to $P$.
Note that $Z_n^{\eta,-\infty}$ is positive for $Q$-almost every $\eta$,
since the random walk has unbounded jumps.

 An important quantity in this model is the so-called free energy defined by
\begin{eqnarray} \label{free energy}
 \begin{split}
\varphi(p,\beta)
&
=\lim_{n\to\infty}\frac{1}{n}\log Z_n^{\eta,\beta}\\
&
=\lim_{n\to\infty}\frac{1}{n}Q[\log Z_n^{\eta,\beta}],
 \end{split}
\end{eqnarray}
 whose existence can be shown by using subadditive ergodic theorem. In~\cite{CFNY}, the continuity property and some asymptotic beahvior of the free energy were studied.

Next, we introduce First Passage Percolation models related to this directed polymer. Denote the (scaled) points where $\eta=0$ by
\begin{equation} \nn
 \omega_p=\sum_{(k,x)\in \N\times\Z^d}(1-\eta(k,x))
 \delta_{(k,s_p x)}, 
\end{equation}
with the scaling factor $s_p=(\log\frac1p)^{1/d}\sim (1-p)^{1/d}$ ($p\uparrow 1$). 
This scaling is natural since $\omega_p$ converges as $p\uparrow 1$ to the Poisson point process $\omega_1$ on $\N\times\R^d$ whose intensity is the product of the counting measure and Lebesgue measure.
With some abuse of notation we will frequently identify $\omega_p$, and more generally any  point measure, with its support. Given a realization of $\omega_p$, we define the minimum passage time from $0$ to $n$ by 
\begin{equation} \label{eq:passtime}
T_n(\omega_p)
= \min\set{\sum_{k=1}^n |x_{k-1}-x_k|^\alpha:
 x_0=0\textrm{ and }\{(k,x_k)\}_{k=1}^n\subset \omega_p}. 
\end{equation} 
This is the directed version of the Howard-Newman's Euclidean FPP model in \cite{HN01}.
Now, a direct application of the subadditive ergodic theorem shows that the limit 
\begin{equation} 
\mu_p=\lim_{n\to\infty}{1 \over n}T_n(\omega_p)
\label{mu_p}
\end{equation} 
exists $Q$-almost surely and equals to $\lim{1 \over n}Q[T_n(\omega_p)]$. The limit $\mu_p$, so-called time constant, is non random. Observe~also that definition (\ref{eq:passtime}) makes perfect sense when $p=1$, yielding a limit $\mu_1$ in (\ref{mu_p}). It is again shown in~\cite{CFNY} that $\mu_p$ is continuous as $p\nearrow{}1$. 

 This FPP is related to the \emph{ground state at $\beta=-\infty$} of the directed polymer introduced above. Namely,
\[\sup\{P(X_i=x_i \text{ for } 1\leq{}i\leq{}n):x_i\in{}\Z^d,\text{ }\eta(i,x_i)=0\}=c^n_1{}\exp\{-c_2s_p^{-\alpha}T_n(\omega_p)\}.\] 
Furthermore, it is shown in~\cite{CFNY} that as
$p \uparrow 1$, the ground state gives dominant contribution to the free energy:
\begin{eqnarray} 
 \varphi (p,-\infty)\sim -c_2\mu_1(1-p)^{-\alpha/d}.
\label{eq:HD}
\end{eqnarray}
\begin{remark}
As one sees above, the subadditive ergodic theorm is useful to show the existence of the free energy and the time constant. However, it prevent us from getting more information about them, such as the continuity. In this paper, we will use a method of concentration of measures to get the rate of convergence in~\eqref{free energy} and~\eqref{mu_p} and derive the continuity results as a corollary. 
\end{remark}
\subsection{Main results}
\hspace{2mm}\\

First, we shall present the results for FPP. The first result is the concentration of the passage time around the mean.
\begin{thm}\label{conc;FPP}
For any $\delta>0$ there exist positive constants $C_1$,$C_2>0$, and $\lambda\in(0,1]$ which are independent of $p$ and $n$ such that for any $n\in\N$,
\begin{equation}
Q(|T_n(\omega_p)-Q[T_n(\omega_p)]|>n^{\frac{1}{2}+\delta})<{}C_1\exp\{-C_2n^{\lambda}\}.\label{con1;FPP}
\end{equation}
\end{thm}
The next result is about the rate of convergence.
It implies that the fluctuation exponent (see \cite{KS91}) is ${1 \over 2}$ or less.
\begin{thm}\label{21}
For any $\chi>1/2$ there exist positive constants $C_0>0$ which are independent of $p$ such that for any $n\in\N$,
\begin{eqnarray}
|n\mu_p-Q[T_n(\omega_p)]|&<&C_0n^{\chi},\label{con3;FPP}
\end{eqnarray}
\end{thm}
Combining \eqref{con1;FPP} and \eqref{con3;FPP}, we have, for any $\chi>1/2$ there exist positive constants $C_0,C_1,C_2,\lambda>0$ which are independent of $p$ such that for any $n\in\N$,
\begin{eqnarray}
Q[|T_n(\omega_p)-n\mu_p|>2C_0n^{\chi})&<&{}C_1\exp\{-C_2n^{\lambda}\}.\label{con2;FPP}
\end{eqnarray}
\begin{cor}\label{conti1}
$\mu_p$ is continuous in $p\in[0,1]$.
\end{cor}
Corollary 1 is a slight extension of Theorem 1.5 in \cite{CFNY} where the continuity is proved only at $p=1$. This follows from Theorem \ref{21} and the continuity of the mean of the passage time in $p$. \\

Next, we move on to the results for directed polymer model. We have the following same properties as in FPP.
\begin{thm}
\label{conc;FE}
In the above setting, for any $q\in[0,1)$, $\delta>0$ and $\beta_0\in\mathbb{R}$, there exist $C_1,C_2>0$, and $\lambda\in(0,1)$ for any $p\in[0,q)$, $\beta\in[-\infty,\beta_0]$, and $n\in\N$,
\begin{eqnarray}
Q(|\log{Z^{\eta,\beta}_n}-Q[\log{Z^{\eta,\beta}_n}]>n^{\frac{1}{2}+\delta})<{}C_1\exp\{-C_2n^{\lambda}\}.\label{con1;FE}
\end{eqnarray}
\end{thm}
\begin{thm}\label{concc;FE}
For any $q\in[0,1)$, $\chi>{1\over 2}$ and $\beta_0\in\mathbb{R}$, there exist $\epsilon,C_0-C_2>0$, and $\lambda\in(0,1)$ for any $p\in[0,q)$, $\beta\in[-\infty,\beta_0]$, $n\in\N$,
\begin{eqnarray}
|n\varphi(p,\beta)-Q[\log{Z^{\eta,\beta}_n}]|&<&C_0n^{\chi},\label{con11;FE}\\
Q[|\log{Z^{\eta,\beta}_n}-n\varphi(p,\beta)|>2C_0n^{\chi})&<&{}C_1\exp\{-C_2n^{\lambda}\}.\label{con2;FE}
\end{eqnarray}
\end{thm}
Next theorem is an extension of Theorem 1.2 in \cite{CFNY} where the continuity is proved only for $\alpha<d$.
\begin{cor}\label{conccc;FE}
$\varphi(p,\beta)$ is jointly continuous on $[0,1)\times[-\infty,\infty)$.
\end{cor}
\begin{remark}
Corollary~\ref{conti1} can in fact be proved by the ``coupling method'' in~\cite{CFNY}. Nevertheless, we think the line of the argument---proving the continuity of a limiting quantity like the time constant via a concentration bound---is of interest. In Corollary~\ref{conccc;FE}, the method of concentration indeed yields a better result than the rather bare-handed approach in~\cite{CFNY}. 
\end{remark}

\begin{remark}\label{choice}
These results still hold even if the definition of $f(k)$ is replaced by $c\exp{\{-V(k)\}}$ where $V$ holds either $0<C_1<V''(x)<C_2$ for any $x$ with some positive constant $C_1$ and $C_2$ or that $V'$ is regularly varying with index $\beta>0$. Indeed, the key lemma below can be shown by the essentially the same way for these choices.
\end{remark}
\subsection{Organization of the paper}
\hspace{2mm}\\

The rest of the paper is organized as follows. Section \ref{sec:concs} is devoted to the proofs of Theorems \ref{conc;FPP} and \ref{conc;FE}. As for FPP, we divide the proof into two parts; $\alpha\le 1$ or $\alpha>1$. We first prove the case $\alpha\le 1$ which is relatively easy. To prove the other case, we need key lemma in Section \ref{KEY}. This lemma is the essential part of this paper. In Section \ref{sec:fluc}, we prove Theorems \ref{21} and \ref{concc;FE}. Finally, we prove Corollaries \ref{conti1} and \ref{conccc;FE} in Section \ref{sec:cont}.\\
\hspace{2mm}\\

\section{Proof of the concentration around the mean}\label{sec:concs}
\subsection{Concentration for FPP with $0<\alpha\le 1$}
\hspace{2mm}\\

\begin{proof}[Proof of~\eqref{con1;FPP} for $0<\alpha\le 1$]
We prove~\eqref{con1;FPP} by using a martingale difference 
method. We introduce a filtration 
\begin{equation}
 \mathcal{G}_m=\sigma(\omega|_{[0,m]\times\R^d})
\end{equation}
and decompose the deviation from the mean into the sum of 
martingale differences as
\begin{equation}
\begin{split}
&T_n(\omega_p)-Q[T_n(\omega_p)]\\
&\quad=\sum_{m=1}^n(Q[T_n(\omega_p)|\mathcal{G}_m]
-Q[T_n(\omega_p)|\mathcal{G}_{m-1}])\\
&\quad=:\sum_{m=1}^n\Delta_m.
\end{split}
\end{equation}
We are going to prove that for some $c>0$
independent of $p$, 
\begin{equation}
 Q\left[\exp \set{c |\Delta_m|^{d/\alpha}}\left| 
\mathcal{G}_{m-1}\right.\right] \le c^{-1} \label{cond-exp}
\end{equation}
$Q$-almost surely. 
Then~\eqref{conc;FPP} 
follows by a concentration inequality for martingales, 
for example, Theorem 1.1 in~\cite{LW09}.
Let us introduce some notation. 
Given two configurations $\omega$ and $\omega'$ and $m\in\N$, 
we define a new configuration by
\begin{equation}
 [\omega, \omega']_m=
\omega|_{[0,m]\times\R^d}+
\omega'|_{[m+1,\infty)\times\R^d}.
\end{equation}
Let $\pi_n^{(m)}$ be a minimizing path for this configuration 
chosen by a deterministic algorithm (if not unique). 
For a time-space point $(k,x)\in \N\times\R^d$, we define
$(k,x)^{*\omega}$ to be a point in $\omega|_{\{k\}\times\R^d}$ 
closest to $(k,x)$.

Now we rewrite the martingale difference as 
\begin{equation}
\Delta_m=\int Q(\d \omega')(T_n([\omega, \omega']_m)
-T_n([\omega, \omega']_{m-1})).
\end{equation}
Then by bounding $T_n([\omega, \omega']_{m-1})$ by the passage time
of the path 
\begin{equation}
\pi_n^{(m)}(1),\ldots,\pi_n^{(m)}(m-1),\pi_n^{(m)}(m)^{*\omega'},
\pi_n^{(m)}(m+1),\ldots,\pi_n^{(m)}(n) 
\end{equation}
and using the fact that $\alpha\le 1$ implies
\begin{equation}
\begin{split}
&|\pi_n^{(m)}(m\pm 1)-\pi_n^{(m)}(m)^{*\omega'}|_1^\alpha\\
 &\quad\le |\pi_n^{(m)}(m\pm 1)-\pi_n^{(m)}(m)|_1^\alpha
 +|\pi_n^{(m)}(m)-\pi_n^{(m)}(m)^{*\omega'}|_1^\alpha, 
\end{split}
\label{triangular}
\end{equation}
we find 
\begin{equation}
 \Delta_m\ge -2\int Q(\d \omega')
 |\pi_n^{(m)}(m)-\pi_n^{(m)}(m)^{*\omega'}|_1^\alpha\label{lb-Delta}
\end{equation}
and similarly, 
\begin{equation}
 \Delta_m\le 2\int Q(\d \omega')
|\pi_n^{(m-1)}(m)-\pi_n^{(m-1)}(m)^{*\omega}|_1^\alpha.\label{ub-Delta}
\end{equation}
Note that $\pi_n^{(m)}(m)$ and $\pi_n^{(m-1)}(m)$ are independent 
of $\omega'|_{\{m\}\times\R^d}$ and $\omega|_{\{m\}\times\R^d}$ 
respectively. Then it follows that
the right-hand side of~\eqref{lb-Delta} is a 
deterministic constant and the right-hand side of~\eqref{ub-Delta} 
is an average of $|(m,x)-(m,x)^{*\omega}|^\alpha$, 
which is easily seen to satisfy~\eqref{cond-exp}. 
\end{proof}
\hspace{2mm}\\
\subsection{Key lemma (The uniform bound for jumps)}\label{KEY}
~\\

As we mentioned before, it is more difficult to prove Theorem~\ref{conc;FPP} in $\alpha>1$ case. Indeed, we need to estimate the change of minimum passage time when we replace the configuration on a section by another one, but we do not have the triangular inequality~\eqref{triangular}. A natural alternative way is to use Taylor's theorem to get
\begin{equation*}
\begin{split}
T_n([\omega, \omega']_m)-T_n([\omega, \omega']_{m-1})\geq{}&-C(|\pi_n^{(m)}(m+1)-\pi_n^{(m)}(m)|^{\alpha-1}_1|\pi_n^{(m)}(m)-\pi_n^{(m)}(m)^{*\omega'}|_1\\
&+|\pi_n^{(m)}(m)-\pi_n^{(m)}(m-1)|^{\alpha-1}_1|\pi_n^{(m)}(m)-\pi_n^{(m)}(m)^{*\omega'}|_1),
\end{split}
\end{equation*}
with some positive constant $C$. Then, the jump size of the optimal path appears in this change and we need to show that it is not too large. Such a geometric property of the optimal path is usually hard to establish but we can show that it has no jumps of polynomial size in $n$. We shall state it in a slightly generalized way that is useful in the study of the directed polymer model.\\

 For any $n$-path $\gamma\in{}(\R^d)^{\{1,\cdots,n\}}$, we write $\gamma(k):=$ the $k$th point of $\gamma,$ and $\Delta\gamma(k):=|\gamma(k)-\gamma(k-1)|_1$, where we put $\gamma(0)=0$ for the convention. Given a configration $\omega$ and $n$-path $\gamma$, we write $\gamma\subset\omega$ if $(i,\gamma(i))\in\omega$ for any $i=1,\cdots,n$. We say $\omega\subset{}\N\times{}\R^d$ has $\theta$-property if 
\begin{eqnarray}
\omega\cap{}(i,x+[0,n^{\theta})^d)\neq{}\emptyset\text{ for any $i\in\N$, $x\in\Z^d$}.
\end{eqnarray}

  Intuitively, $\theta$-property means configration has no big vacant regions. As we shall see in Lemma \ref{approx}, the assumption of $\theta$-property have a little influence on the minimum passage time.
\begin{Def}
  For any $n$-path $\gamma\in\R^{\{1,\cdots,n\}}$, $T_n(\gamma)$ is defined by
  \[T_n(\gamma)=\sum^{n}_{i=1}\Delta{}\gamma(i)^{\alpha}.\]
\end{Def}
The next lemma is the key of all results in this paper. Roughly speaking, for any polymer, either it has no big jumps or we can find another polymer with a smaller passage time. 
\begin{lem}\label{28}
Suppose $\alpha>{}1$. For any $\zeta>0$, there exist $\theta>0$ and $N=N(\theta)\in\mathbb{N}$ such that for any $n>N(\theta)$, $s\in\{1,\cdots,n\}$, an n-path $\gamma$ and a configuration $\omega'$ which has $\theta$-Property, either of the following holds;
\begin{enumerate}
 \item $\max\{\Delta\gamma(s),\Delta\gamma(s+1)\}\leq{}n^{\zeta}$,
 \item There exists an $n$-path $\gamma'$ and $k>0$ such that $\gamma(i)=\gamma'(i)$ for any $i\notin[s,s+k-1]$, $(i,\gamma'(i))\in\omega'$ for $i\in[s,s+k-1]$ and 
 \[T_n(\gamma')+(k+1)n^{\theta}\leq{}T_n(\gamma).\]
\end{enumerate}
\end{lem}

The proof of Lemma 1 is not long but a bit complicated. 
Let us explain the idea of the proof in the case that $\gamma$ is the minimizing
path for $\omega'$. Let $A_0$ be the point where next jump is larger than $n^\zeta$
and introduce a sequence of large numbers satisfying 
$n^\zeta\gg\mathcal{L}_1\gg\mathcal{L}_2\gg\cdots$. 
\begin{figure}[here]
\fbox{
   \includegraphics[width=40mm]{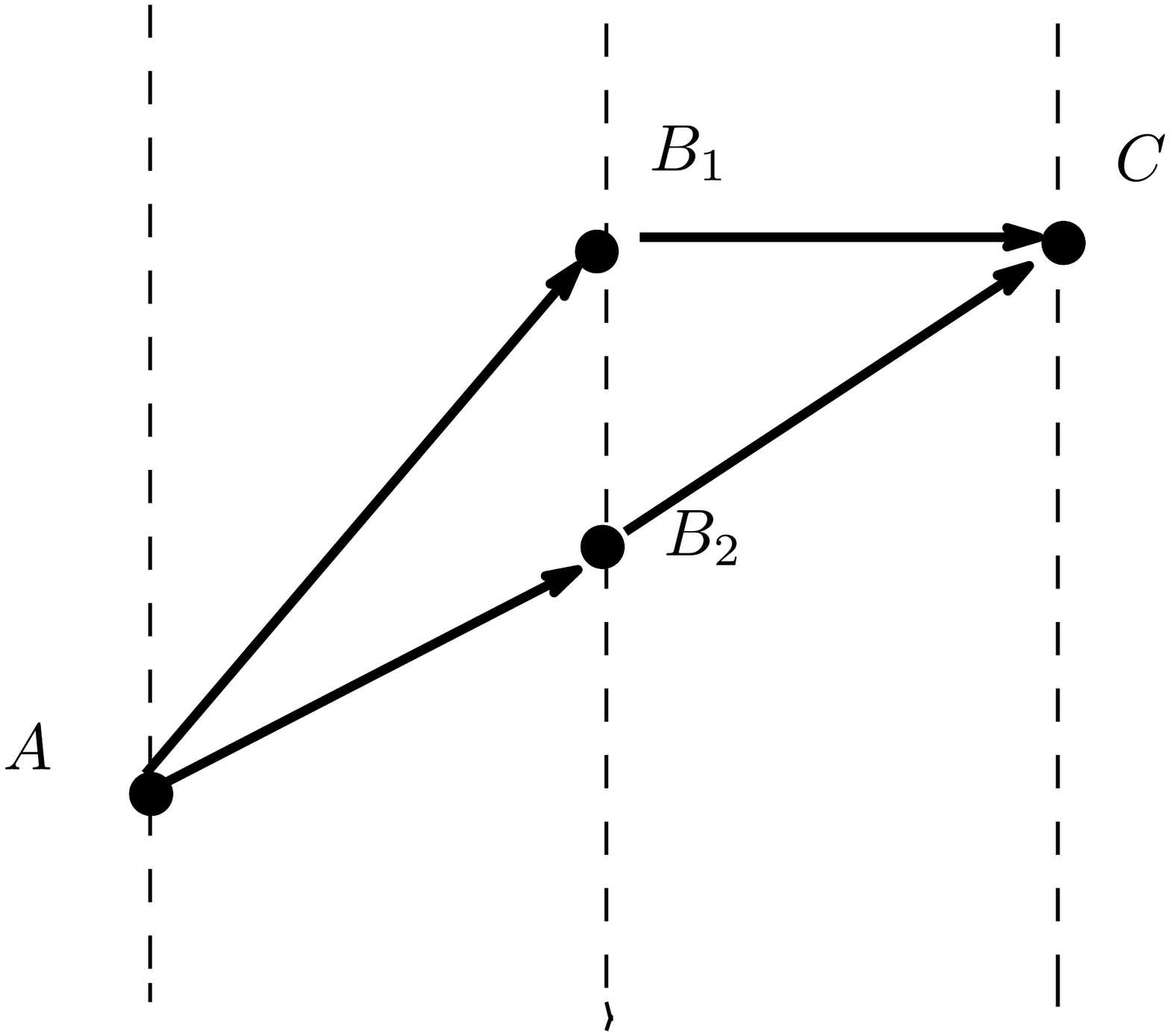}
}
\hspace{5mm}
\fbox{
  \includegraphics[width=40mm]{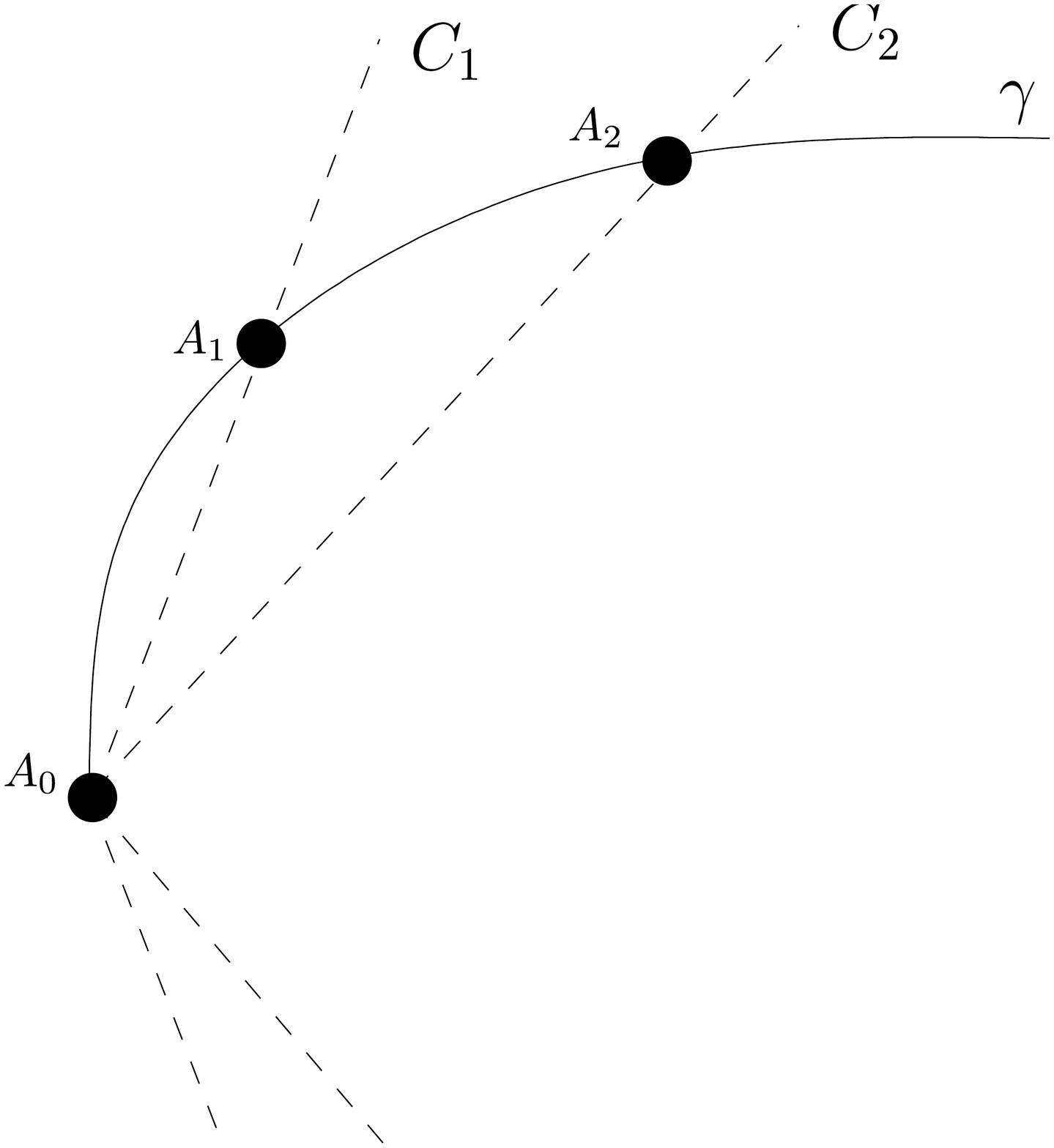}
}
\caption{}
Left: Passage time of $AB_2{}C$ is smaller than that of $AB_1{}C$ when $\alpha>1$.\\
Right: Drawing cones, we investigate $\gamma$ between $A_0$ and $A_k$ inductively.
  \label{fig:two}
\end{figure}

We draw cones $C_k$ with slope $\mathcal{L}_k$ starting at $A_0$ and let $A_k$ be the first point where the optimal path $\gamma$ touches it. 
Then we draw a straight line between $A_0$ and $A_k$. Thanks to the 
$\theta$-property, we can find a path $\gamma_k'\subset \bar{\omega}$ close
to this straight line. Due to the convexitity of $|\cdot|^{\alpha}_1$, average jumps are better (See Figure \ref{fig:two} above). This yields that, if we forget about the restriction $\subset\bar\omega$, the straight line connecting 
$A_0$ and $A_k$ is \emph{the optimal path} when $\alpha>1$. 
On the other hand, by definition, $\gamma$ has a smaller passage time than 
$\gamma_k'$ does. Then it is natural to expect that most of the jumps of 
$\gamma$ between $A_0$ and $A_k$ are close to $\mathcal{L}_k$ in size. 
But in the first $A_0\to A_{k-1}$ segment, most of the jumps are close to
$\mathcal{L}_{k-1}$ for the same reason and hence far from $\mathcal{L}_k$. 
This implies that the duration of $A_0\to A_k$ is much longer than that of
$A_0\to A_{k-1}$. Based on this observation, we can prove that $A_k$ goes
beyond the $n$-th section before $\mathcal{L}_k$ becomes small. 
But this implies that the optimal path $\gamma$ stays outside a cone of
very large angle until $n$. This is very unlikely because optimal path has too much passage time between $A_0$ and $n$-th section and one can indeed derive a
contradiction. 
\begin{proof}
Let $m\in\N$ such that $m>5/\zeta$, and we take $\theta>0$ sufficiently small and $N=N(\theta)\in\mathbb{N}$ such that sufficiently large. (For example, if  $\theta>0$ and $N$ such that \[2^5m\alpha\theta<\zeta ~~\text{and}~~ N^{\min\{\alpha-1,1\}\theta}>2^{\alpha+4}\alpha^2(\alpha-1)^{-1}+(2^{\frac{1}{\alpha}}-1)^{-1}),\]one can see that the following proof works.) Suppose that neither (i) nor (ii) holds and we shall derive a contradiction. This means we suppose the following two conditions;\\

\begin{itemize}
 \item[(i')]$\max\{\Delta\gamma(s),\Delta\gamma(s+1)\}>n^{\zeta}$,
 \item[(ii')]For any an $n$-path $\gamma'$ and $k>0$ such that $\gamma(i)=\gamma'(i)$ for $i\notin[s,s+k]$, and $(i,\gamma'(i))\in\omega'$ for $i\in[s,s+k]$, $T_n(\gamma')+(k+1)n^{\theta}>{}T_n(\gamma)$.
\end{itemize}

Due to (i'), either $\Delta\gamma(s)>{}n^{\zeta}$ or $\Delta\gamma(s+1)>n^{\zeta}$. Let $\tau$ be such that $n^{\tau}=\Delta\gamma(s)$. Then, taking a path $\gamma'$ such that $\gamma'(i)=\gamma(i)$ for any $i\neq{}s$, $(s,\gamma'(s))\in\omega'$ and $|{\gamma(s-1)+\gamma(s+1) \over 2}-\gamma'(s)|_1\le dn^{\theta}$, we have 
\begin{equation*}
T_n(\gamma)-T_n(\gamma')\geq{}\Delta\gamma(s)^{\alpha}+\Delta\gamma(s+1)^{\alpha}-2\left({\Delta\gamma(s+1)+\Delta\gamma(s)\over{}2}+dn^{\theta}\right)^{\alpha}.
\end{equation*}
If $2\tau<\zeta$, this is further bounded from below by $2n^{\theta}$, which contradicts (ii'). Hence we have $2\tau\geq{}\zeta$.\\

For any $k\leq{}m$, let
\begin{eqnarray}
\ell_{k}:=\inf\{\ell\in[s,n]\cap\mathbb{N}:\frac{|\gamma(\ell)-\gamma(s-1)|_1}{\ell-s+1}\leq{}n^{\tau-(2k-1)\theta}\}.\label{1}
\end{eqnarray}
By the $\theta$-property of $\omega'$, there exists a path $\gamma'$ such that $\gamma(i)=\gamma'(i)$ for any $i\leq{}s-1$ and $(i,\gamma(i))\in\omega'$, $\Delta\gamma(i)\leq{}2dn^{\theta}$ for any $i\geq{}s$. If the range of infimum of (\ref{1}) is empty, then 
\begin{eqnarray}
T_n(\gamma)-T_n(\gamma')-(n-s+1)n^{\theta}\geq{}(n-s+1)(n^{\tau-(2k-1)\theta}-2dn^{\theta}-n^{\theta})\geq{}0,
\end{eqnarray}
and (ii) follows. Hence, we can suppose the range is not void for $k\leq{}m$.\\

Let $\bar{\gamma}_k$ be a straight line drawn between $(s-1,\gamma(s-1))$ and $(\ell_k, \gamma(\ell_k))$. We write a slope $\frac{|\gamma(\ell_k)-\gamma(s-1)|_1}{\ell_k-s+1}$ of $\bar{\gamma}_k$ as $\mathcal{L}_k$. Since $\omega'$ has $\theta$-property, there exists a n-path $\gamma'_k$ such that $(i,\gamma(i))\in\omega'$, $|\bar{\gamma}_k(j)-\gamma'_k(j)|_1\leq{}dn^{\theta}$ for all $s\leq{}j<\ell_k$, $\gamma'_k(\ell)=\gamma(\ell)$ for $\ell\leq{}s-1$ or $\ell\geq\ell_k$. Set $R_k:=\ell_k-s$, $D_k:=\{0\leq{}i\leq{}R_k:\Delta\gamma(s+i)\geq{}n^{\tau-2k\theta}\}$, $D_k^c:=\{0\leq{}i\leq{}R_k\}\setminus{}D_k$.\\

We will show that $\#D_{K+1}\geq{}n^{\tau/2}\#D_K$ for any $K<m$.
By iteration, this yields $\#D_{m}\geq{}n^{\tau{}m/2}\geq{}n^{5/4}$, 
a contradiction. 
\\

Set $\delta_{j,K}:=\Delta\gamma(s+j)-\mathcal{L}_{K}$ and $\delta'_{j,K}:=\Delta\gamma_{K}'(s+j)-\mathcal{L}_{K}$.
We note that $\sum^{R_K}_{i=0}\delta_{i,K}, \sum{}\delta'_{i,K}\geq{}0$, and $|\delta'_{j,K}|\leq{}2dn^{\theta}$.\\

We first prove that
\begin{eqnarray*}
n^{\tau-(2K-1)\theta}-n^{\tau-2K\theta}\leq{}\mathcal{L}_K\leq{}n^{\tau-(2K-1)\theta},
\end{eqnarray*}
that is, $\mathcal{L}$ does not overshoot. Our strategy is to show that either the last jump $\Delta\gamma(s+R_K)$ is small or $R_K$ is large. We choose $k'$ such that $n^{k'}=\Delta\gamma(s+R_K)$. By (ii'), 
\begin{eqnarray}
  \sum^{R_K}_{i=0}\{(\mathcal{L}_{K}+\delta'_{i,K})^\alpha-\mathcal{L}_{K}^\alpha+n^{\theta}\}
  \geq{}\sum^{R_K}_{i=0}\{(\mathcal{L}_{K}+\delta_{i,K})^\alpha-\mathcal{L}_{K}^{\alpha}-\alpha\mathcal{L}_{K}^{\alpha-1}\delta_{i,K}\}.\label{ohon}
\end{eqnarray}
The left hand side of \eqref{ohon} is bounded from above by
\begin{eqnarray*}
  \sum^{R_K}_{i=0}\{(\mathcal{L}_{K}+\delta'_{i,K})^\alpha-\mathcal{L}_{K}^\alpha+n^{\theta}\}\le \sum^{R_K}_{i=0}\{4d\alpha{}\mathcal{L}_{K}^{\alpha-1}n^{\theta}+n^{\theta}\}\le R_Kn^{\tau(\alpha-1)+2\theta},
\end{eqnarray*}
where we have used $|a+b|^{\alpha}-|a|^{\alpha}\le \alpha|b-a|(|a|+|b|)^{\alpha-1}$ for $a,b\in\R$ in the first inequality. 
On the other hands, the right hand side is bounded from below by
\begin{eqnarray*}
&~&{}\sum^{R_K}_{i=0}\{(\mathcal{L}_{K}+\delta_{i,K})^\alpha-\mathcal{L}_{K}^{\alpha}-\alpha\mathcal{L}_{K}^{\alpha-1}\delta_{i,K}\}\\
&\geq&{}\left((\mathcal{L}_{K}+\delta_{R_K,K})^\alpha-\mathcal{L}_{K}^{\alpha}-\alpha\mathcal{L}_{K}^{\alpha-1}\delta_{R_K,K}\right)+((\mathcal{L}_{K}+\delta_{0,K})^\alpha-\mathcal{L}_{K}^{\alpha}-\alpha\mathcal{L}_{K}^{\alpha-1}\delta_{0,K})\\
  &\geq&\max\{\frac{n^{k'\alpha}}{2},\frac{n^{\tau\alpha}}{2}\},
\end{eqnarray*}
where we have used the fact that 
\begin{eqnarray}
(b-a)^{\alpha}-b^{\alpha}+\alpha{}ab^{\alpha-1}\geq{}0 \text{\hspace{4mm}for any }0<b\text{ and }a\leq{}b\label{kippo2}
\end{eqnarray}
 in the first inequality and $\delta_{0,K}=\Delta\gamma(s)-\mathcal{L}_{K}\geq{}n^{\tau}/2\gg{}n^{\tau-(2K-1)\theta}\geq\mathcal{L}_K$ in the second inequality. Consequently, we have
 \[
R_K\geq{}{1 \over 2}\max\{n^{k'\alpha-(\alpha-1)\tau-2\theta},n^{\tau-2\theta}\}\geq{}\max\{n^{k'-3\theta},n^{\tau-3\theta}\}.\]
This yields
\begin{eqnarray*}
\mathcal{L}_K&\geq&{}\frac{1}{R_K+1}(|\gamma(\ell_K-1)-\gamma(s-1)|_1-\Delta\gamma(s+R_K))\\
&\geq&n^{\tau-(2K-1)\theta}(1-\frac{1}{R_K+1})-n^{3\theta}\\
&\geq&n^{\tau-(2K-1)\theta}-n^{\tau-2K\theta}.
\end{eqnarray*}

We shall get back to the main proof. Due to (ii'),
\begin{eqnarray}
\begin{split}
\sum^{R_{K+1}}_{i=0}\{(\mathcal{L}_{K+1}+\delta_{i,K+1})^{\alpha}-\mathcal{L}_{K+1}^{\alpha}-\alpha\mathcal{L}_{K+1}^{\alpha-1}\delta_{i,K+1}\}\\
\leq{}\sum^{R_{K+1}}_{i=0}\{(\mathcal{L}_{K+1}+\delta'_{i,K+1})^{\alpha}-\mathcal{L}_{K+1}^{\alpha}+n^{\theta}\}.
\end{split}\label{HaS}
\end{eqnarray}
When $i\in{}D^c_{K+1}$, we have $-\delta_{i,K+1}\geq\mathcal{L}_{K+1}-n^{\tau-2(K+1)\theta}$ and hence
\begin{equation}
\begin{split}
(\mathcal{L}_{K+1}+\delta_{i,K+1})^{\alpha}-\mathcal{L}_{K+1}^{\alpha}-\alpha\mathcal{L}_{K+1}^{\alpha-1}\delta_{i,K+1}\hspace{30mm}\\
\geq{}-\mathcal{L}_{K+1}^{\alpha}+\alpha\mathcal{L}_{K+1}^{\alpha-1}(\mathcal{L}_{K+1}-n^{\tau-2(K+1)\theta})
\geq\frac{\alpha-1}{2}\mathcal{L}_{K+1}^{\alpha}.
\end{split}\label{kippo}
\end{equation}
We can estimate the left hand side of (\ref{HaS}) further bounded from below as
\begin{eqnarray*}
\text{LHS of (\ref{HaS})}&\geq&\sum_{i\in{}D_{K}}\{(\mathcal{L}_{K+1}+\delta_{i,K+1})^{\alpha}-\mathcal{L}_{K+1}^{\alpha}-\alpha\mathcal{L}_{K+1}^{\alpha-1}\delta_{i,K+1}\}+\frac{\alpha-1}{2}\sum_{i\in{}D_{K+1}^c}\mathcal{L}^{\alpha}_{K+1},
\end{eqnarray*}
where we have used \eqref{kippo2}, $D_K\subset D_{K+1}$ and \eqref{kippo}. As for the right hand side, we have
\begin{eqnarray*}
\text{RHS of (\ref{HaS})}&\leq&\sum_{i\in{}D_{K+1}}\{(\mathcal{L}_{K+1}+\delta'_{i,K+1})^{\alpha}-\mathcal{L}_{K+1}^{\alpha}+n^{\theta}\}+\frac{\alpha-1}{2}\sum_{i\in{}D_{K+1}^c}\mathcal{L}^{\alpha}_{K+1}\\
&\leq&\sum_{i\in{}D_{K+1}}4d\alpha{}n^{(\alpha-1)(\tau-(2K+1)\theta)+\theta}+\frac{\alpha-1}{2}\sum_{i\in{}D_{K+1}^c}\mathcal{L}^{\alpha}_{K+1}
\end{eqnarray*}
 by the definition of $\mathcal{L}_{K+1}$. 
From these, we get 
\begin{eqnarray*}
  4d\alpha\#D_{K+1}n^{(\alpha-1)(\tau-(2K+1)\theta)+\theta}
&=&4d\alpha\sum_{i\in{}D_{K+1}}n^{(\alpha-1)(\tau-(2K+1)\theta)+\theta}\\
&\geq&\sum_{i\in{}D_{K}}\{(\mathcal{L}_{K+1}+\delta_{i,K+1})^{\alpha}-\mathcal{L}_{K+1}^{\alpha}-\alpha\mathcal{L}_{K+1}^{\alpha-1}\delta_{i,K+1}\}\\
  &\geq&\#D_{K}\min_{i\in{}D_K}\{(\mathcal{L}_{K+1}+\delta_{i,K+1})^{\alpha}-\mathcal{L}_{K+1}^{\alpha}-\alpha\mathcal{L}_{K+1}^{\alpha-1}\delta_{i,K+1}\}\\
&\geq&{}\frac{1}{2}\#D_Kn^{\alpha(\tau-2K\theta)},
\end{eqnarray*}
where we have used $\delta_{i,K+1}$ is much larger than $\mathcal{L}_{K+1}$ for
$i\in D_K$ in the last inequality. Rearranging yields 
$\#D_{K+1}\ge n^{\tau/2}\#D_K$ for a sufficiently small $\theta$ 
as desired. 
\end{proof}
\hspace{2mm}\\
\subsection{Concentration for FPP with $\alpha>1$}
~\\

In this subsection, we prove Theorem \ref{conc;FPP} for $\alpha>1$. We fix a small $\theta>0$ and define 
 \begin{equation}\label{2.2}
 \bar{\omega}=\omega+\sum_{(k,x)\in \N\times n^\theta\Z^d}
 1_{\{\omega(\{k\}\times (x+[0,n^\theta)^d))=0\}}\delta_{(k,x)},
 \end{equation}
that is, 
when we find a large vacant box, we add an $\omega$-point 
artificially at a corner. 
We first recall Lemma 3.3 in~\cite{CFNY} which shows that $T_n(\omega_p)$ and $T_n(\bar\omega_p)$  are essentially the same.
\begin{lem}
\label{approx}
There exists $C_4>0$ such that for sufficiently large $n\in\N$, 
\begin{eqnarray}
\begin{split}
&\max\{Q(T_n( \omega_p)\neq T_n( \bar\omega_p)),
 Q[|T_n( \omega_p)-T_n( \bar\omega_p)|]\}\\
&\quad \le \exp\{-C_4n^{d\theta}\}.
\end{split}
\end{eqnarray}
\end{lem}
\begin{proof}
This is proved only for $\omega_1$ in~\cite{CFNY} but the same proof 
works for general $\omega_p$.
\end{proof}

\begin{proof}[Proof of~\eqref{con1;FPP} for $\alpha> 1$]
Fix $\delta>0$. Let us denote by $\bar\omega_p^{(m)}$ the point process obtained by replacing its $\{m\}\times\R^d$-section by another configuration
$\bar\omega'$. 
We are going to use the so-called entropy method and it 
requires a bound on
\begin{eqnarray}
\sum_{m=1}^n\left(\sup_{\omega'}|T_n(\bar\omega_p^{(m)})
-T_n(\bar\omega_p)|\right)^2,
\label{ES}
\end{eqnarray}
where the supremum is taken over all configrations $\omega'$.

Let $\pi_n$ be a minimizing path for $\bar{\omega}_p$. Note that this depends implicitly on $\theta$ through \eqref{2.2}.
\begin{lem}\label{2}
Suppose $\alpha>1$. For any $\zeta>0$, there exists $\theta\in(0,\zeta)$ and \\
$N=N(\theta)\in\mathbb{N}$  such that for any $n>N$ and $1\le i\leq{}n$,
\[\Delta\pi_n(i)\leq{}n^{\zeta}.\]
\end{lem}
\begin{proof}
  Let $\omega'=\bar{\omega}_p$ $\gamma=\pi_n$ in Lemma \ref{28}. Then for any $s\in\{1,\cdots,n\}$, either (i) or (ii) holds. If (ii) holds, it contradicts that $\pi_n$ is a minimizing path. It follows that (i) holds and we get the desired conclusion.
\end{proof}

By using Lemma \ref{2}, we can bound the summands of $($\ref{ES}$)$ from above by $n^{o(1)}$ and consequently \eqref{ES} itself by $n^{1+o(1)}$ a.s.
\begin{lem}
\label{mart-diff}
For any $\zeta>0$, there exists two positive constants $\theta,N$ such that for all  $n\ge{}N$ and $\omega_p$, 
\begin{equation}
  |\Delta_m|\le 4{}n^{\zeta\alpha},
\end{equation}
  where $\Delta_m:=\sup_{\omega'}|T_n(\bar\omega_p^{(m)})
-T_n(\bar\omega_p)|$.
\end{lem}
\begin{proof}
We prove the lower bound for $\Delta_m$. 
The upper bound can be proved similarly.
Fix some $\omega_p$ and $\omega'_p$. Let $\pi_n$ denote a minimizing path for
$T_n(\bar\omega_p)$, 
chosen by a deterministic algorithm if not unique. 
We define a new point $\tilde\pi_n^{(m)}(m)$ as a point in
$\bar\omega_p'|_{\{m\}\times\R^d}$ satisfying 
\begin{equation}
|\pi_n(m)-\tilde\pi_n^{(m)}(m)|_1\le dn^\theta.
\label{tilde2}
\end{equation}
Then we can bound 
$T_n(\bar\omega_p^{(m)})$ from above 
by the passage time of the path 
\begin{equation}
\pi_n(1),\ldots,\pi_n(m-1),\tilde\pi_n^{(m)}(m),
\pi_n(m+1),\ldots,\pi_n(n).
\end{equation}
 By using Lemma~\ref{2} together with \eqref{tilde2}, for sufficiently small $\theta$ and large n, we get
\begin{equation}
\begin{split}
&T_n(\bar\omega_p^{(m)})-T_n(\bar\omega_p)\leq{}T_n(\tilde{\pi}^{(m)}_n)-T_n(\pi_n)\\
&\quad = |\tilde{\pi}^{(m)}_n(m-1)-\tilde\pi_n^{(m)}(m)|_1^\alpha +|\tilde\pi_n^{(m)}(m)-\tilde{\pi}^{(m)}_n(m+1)|_1^\alpha\\
&\quad\le 2(n^{\zeta}+dn^{\theta})^{\alpha}<4n^{\zeta\alpha}
\label{mart-lb}
\end{split}
\end{equation}
as desired. The reverse inequality can be proved by a similar way.
\end{proof}

If we  take $\zeta$ sufficiently small, $($\ref{ES}$)$ is less than $Cn^{1+\delta}$.
Then, Theorem~6.7 in~\cite{BLM13} yields
\begin{eqnarray} \nn
Q\left( |T_n(\bar\omega_p)-Q[T_n(\bar\omega_p)]|<n^{\frac{1}{2}+\delta}
\right)
\le \exp\{-C_2 n^{1-\delta}\}.
\end{eqnarray} 
Lemma~\ref{approx} shows that this remains valid with $\bar\omega_p$
replaced by $\omega_p$ and $\exp\{-C_4n^{d\theta}\}$ added to the 
right-hand side. Finally, This leads us to
\begin{eqnarray} \nn
Q( |T_n(\omega_p)-n\mu_p|>n^{\frac{1}{2}+\delta})
\le \exp\{-C_2n^{1-\delta}\}+\exp\{-C_4n^{d\theta}\},
\end{eqnarray}
which implies \eqref{con1;FPP}.
\end{proof}
Combined the proof of \eqref{con1;FPP} with Borel-Cantelli Lemma, we also have the following corollary.
\begin{cor}
For any $\zeta>0$, as $N\to\infty$,
\[Q(\text{For any $n\geq{}N$ and optimal path }\pi_n\text{ for $T_n(\omega_p)$ and }i\in\{1,\cdots,n\},\text{ }\Delta\pi_n(i)\le n^{\zeta})\to{}1.\]
\end{cor}
\hspace{2mm}\\
\subsection{Concentration of the free energy}
\hspace{2mm}\\
In this subsection, we prove Theorem~\ref{conc;FE}. 
By the relation
\begin{equation}
\log{Z^{\eta,\beta}_n}=\beta{}n+\log{Z^{1-\eta,-\beta}_n},\label{selfaj}
\end{equation}
we have only to consider the case $\beta\in[-\infty,0]$. 
To simplify the notation, we write $Z_{n}(\omega)$ instead of $Z^{\beta,\eta}_{n}$ where $\omega$ is defined as 
$$\omega:=\sum_{(k,x)\in\N\times\Z^d}(1-\eta(k,x))\delta_{(k,x)}.$$
We also write $M:=1+\alpha^{-1}$ as before. For a given configuration $\eta$, or equivalently $\omega$, and an $n$-path $\gamma$, we denote the free energy per path by 
\[F_n(\gamma;\omega)
:=c_2\sum^n_{i=1}\Delta\gamma_i^{\alpha}-\beta\sum^n_{i=1}\eta(i,\gamma(i))
{=c_2T_n(\gamma)-\beta H_n^\eta(\gamma)}.\]
{We assume $\gamma$ starts at the origin throughout this subsection.}
Then we can write the partition function as
\[Z_{n}(\omega)=c_1^n\sum_{\gamma}e^{-F_n(\gamma;\omega)}.\]
We define $\bar{\omega}$ by \eqref{2.2} and the restricted partition function by
\[\tilde{Z}_{n}(\omega):=c^n_1\sum_{{\gamma: T_n(\gamma)\leq{}{}n^{1+2\alpha\theta}}}e^{-F_n(\gamma;\omega)}.\]
First, we bound the difference of partition functions $Z_n(\omega)$ and $Z_n(\bar{\omega})$. 
We begin with the following tail bounds.
\begin{lem}
~\\
\label{greedy}
\begin{enumerate}
\item There exists $C_0>0$ independent of $p\in(0,1]$ such that 
 for all $n\in\N$ and $m>C_0n$,
\begin{equation}
  Q(T_n(\omega_p)>m)\le \exp\{-m^{1\wedge {d \over \alpha}}/C_0\}. 
\end{equation}
 \label{upper-tail}
\item There exists $C_1>0$ and $N\in\N$ such that for all $n>N$,
\begin{equation*}
 Z_n(\omega)-\tilde{Z}_n(\omega)
 \leq{}e^{-C_1{}n^{1+2\alpha\theta}}.
\end{equation*}
\end{enumerate}
\end{lem}
\begin{proof}
(i) See Lemma 3.2 in \cite{CFNY}. Although $\omega_p$ is replaced by $\omega_1$ in \cite{CFNY}, the proof is essentially the same.\\
(ii) We bound $c_1^{-n}(Z_n(\omega)-\tilde{Z}_n(\omega))$ as
\begin{equation}
\begin{split}
 \sum_{T_n(\gamma)>n^{1+2\alpha\theta}}e^{-F_n(\gamma;\omega)}
&=\sum_{k\geq{}2}\sum_{\gamma:{n^{1+k\alpha\theta}<{}T_n(\gamma)
\leq{}n^{1+(k+1)\alpha\theta}}}e^{-c_2 T_n(\gamma)}\\
 &\leq\sum_{k\geq{}2}\sum_{\gamma: T_n(\gamma)
\leq{}n^{1+(k+1)\alpha\theta}}e^{-c_2{}n^{1+k\alpha\theta}}.
\label{decomp}
\end{split}
\end{equation}
If $T_n(\gamma)\leq{}{}n^{1+(k+1)\alpha\theta}$, then $\gamma$ has
no jump larger than $n^{\frac{1}{\alpha}+(k+1)\theta}$ and thus
$\max_{1\le i \le n}|\gamma(i)|\leq n^{M+(k+1)\theta}$. 
This yields the bound
\begin{equation*}
\begin{split}
  \#\{\gamma:{T_n(\gamma)
 \leq{}n^{1+(k+1)\alpha\theta}}\}
&\le n^{d(M+(k+1)\theta)n}.
\end{split}
\end{equation*}
Substituting this into~\eqref{decomp}, we obtain
\begin{equation*}
Z_n(\omega)-\tilde{Z}_n(\omega)
\leq
c_1^n\sum_{k\geq{}2}n^{d(M+(k+1)\theta)n}e^{-c_2{}n^{1+k\alpha\theta}}
\leq{}e^{-C_1{}n^{1+2\alpha\theta}}
\end{equation*}
for sufficiently large $n$.  
\end{proof}

\begin{lem}\label{hod1} 
For any $q\in[0,1)$ and $\theta>0$, there exists $\lambda$, $C_1>0$ and $N\in\mathbb{N}$ such that for any $n>N$ and $p\in{}[0,q)$,\\
  \begin{enumerate}
  \item $Q(|\log Z_{n}(\omega)-\log Z_{n}(\bar{\omega})|>\log 2)\leq{}e^{-C_1n^{\lambda}}$,
 \item $0\leq{}Q[\log{Z_{n}(\bar{\omega})}]-Q[\log{Z_{n}(\omega)}]\leq{}1$.
\end{enumerate} 
\end{lem}
\begin{proof}
Since $Z_n(\omega)\ge c_1^ne^{-c_2s_p^{-\alpha}T_n(\omega_p)}$, 
Lemma~\ref{greedy}-(i) implies that there exist $C_2, C_3>0$ such that
\begin{equation}
 \begin{split}
&Q(Z_{n}(\omega)\leq{}c^n_1{}e^{-c_2{}n^{1+\alpha\theta}})
\leq{}e^{-C_2 n^{(1\wedge \frac{d}{\alpha})(1+2\alpha\theta)}},\\
&Q[|\log{Z_{n}(\omega)}|^2]\leq{}C_3n^2.\label{tail-and-2nd}
\end{split}
\end{equation}
Also it is plain to see (from Lemma 3.3 in \cite{CFNY}) that 
\begin{equation}
 Q(\omega=\bar{\omega}\text{ on }[-n^{M+2\theta},n^{M+2\theta}]^d
\times[0,n])\le e^{-C_4n^{d\theta}}. 
\label{coincidence}
\end{equation}
Thus to prove (i), it suffices to show $Z_n(\bar\omega)<2Z_n(\omega)$ 
under the two conditions: 
$Z_{n}(\omega)\geq{}c^n_1{}e^{-c_2{}n^{1+\alpha\theta}}$ and 
$\omega=\bar{\omega}$ on $[-n^{{M+2\theta}},n^{{M+2\theta}}]^d\times[0,n]$.
(Note that $Z_n(\omega)\le Z_n(\bar{\omega})$.)
Observe that if $\gamma$ exits 
$[-n^{{M+2\theta}},n^{{M+2\theta}}]^d\times[0,n]$, then 
$T_n(\gamma)>n^{1+2\alpha\theta}$ as it must contain 
a jump larger than $n^{\frac{1}{\alpha}+2\theta}$. Therefore
under the above conditions, 
\begin{equation*}
\begin{split}
  Z_{n}(\bar{\omega})-Z_{n}(\omega)
& \leq{}c^n_1{}\sum_{\gamma: \max_{1\le k \le n}|\gamma(k)|_\infty{}>
 n^{M+2\theta}}{}e^{-F_n(\gamma;\omega)}\\
& \leq{}Z_n(\omega)-\tilde{Z}_n(\omega)\\
& \leq{}e^{-C_1n^{1+2\alpha\theta}}\leq{}Z_n(\omega),
\end{split}
\end{equation*}
where we have used Lemma~\ref{greedy}-(ii). 
This in turn implies
\begin{equation*}
\begin{split}
&Q[\log{Z_{n}(\bar{\omega})}]-Q[\log{Z_{n}(\omega)}]\\
 &\quad\leq{}Q(Z_{n}(\omega)\geq{}c^n_1{}e^{-c_2{}n^{1+\alpha\theta}} 
\text{ and } \omega=\bar{\omega} \text{ on } 
[-n^{M+2\theta},n^{M+2\theta}]^d\times[0,n])\log{2}\\
 &\qquad+Q[|\log{Z_{n}(\omega)}|;Z_{n}(\omega)<{}c^n_1e^{-c_2{}
n^{1+\alpha\theta}}\text{ or }\omega\neq\bar{\omega} \text{ on } 
[-n^{M+2\theta},n^{M+2\theta}]^d\times[0,n]]\\
 &\quad\leq{}1,
\end{split}
\end{equation*}
where in the last line we have used \eqref{tail-and-2nd}, 
\eqref{coincidence} and the Schwarz inequality to bound the second term. 
\end{proof}
From Lemma \ref{hod1}, if we show the concentration for $\bar{\omega}$, 
that is, for any $\delta>0$ there exists $\lambda\in(0,1)$ and 
$C_0, C_1>0$ such that 
\begin{equation}
 Q(|\log{Z_{n}(\bar{\omega})}-Q[\log{Z_{n}(\bar{\omega})}]|>n^{\frac{1}{2}+\delta})\leq{}C_0e^{-C_1n^{\lambda}},
\label{CON}
\end{equation}
we can also see the concentratioin for $\omega$.
%
%
We will use the entropy method as in the FPP case and the following two lemmas 
correspond to Lemma~\ref{2} and~\ref{mart-diff}.

\begin{lem}\label{6}
  Let $s\in\{1,\cdots,n\}$. We take configurations $\omega,\omega'$ such that $\omega|_{\{\ell\}\times\R^d}=\omega'|_{\{\ell\}\times \R^d}$ for any $\ell\neq{}s$. Then, for any path $\gamma$, there exists $\gamma'$ such that either;
  \begin{enumerate}
\item for any $i\neq{}s$, $\gamma(i)=\gamma'(i)$, and $F_n(\gamma';\bar{\omega}')\leq{}c_2n^{(\alpha-1)\zeta+3\theta}+F_n(\gamma;\bar{\omega})$ or
\item there exists $k>0$ such that for any $i\notin[s,s+k]$, $\gamma(i)=\gamma'(i)$, and \[F_n(\gamma';\bar{\omega}')+c_2(k+1)n^{\theta}\leq{}F_n(\gamma;\bar{\omega}),\]
  \end{enumerate}
  where the constant $c_2$ comes from (\ref{f(k)}).
\end{lem}
\begin{proof}
 When $\alpha\leq{}1$, it is easy to prove that (i) holds for any $\gamma$. Suppose $\alpha>1$ and we apply the Lamma \ref{28} with $\omega'=\bar{\omega}'$. It suffices to show that (i) and (ii) in Lemma \ref{28} lead to (i) and (ii) in Lemma \ref{6}, respectively. As for (i), we first construct a polymer $\gamma'$ such that $\gamma'(i)=\gamma(i)$ for $i\neq{}s$, $(s,\gamma'(s))\in\bar{\omega}'|_{\{m\}\times{}\R^d}$, and $|\gamma'(s)-\gamma(s)|_1\le dn^{\theta}$. Then, it is easy to check that this $\gamma'$ satisfies the condition (i) in this lemma. On the other hand, we can easily check the claim for (ii) because of the assumption of $\beta\in[-\infty,0]$, which is declared at the beginning of this subsection.
\end{proof}
\begin{lem}\label{5}
Let $\omega$ and $\omega'$ be as in Lemma~\ref{6}. For any $\delta>0$, there exist $\theta>0$ and $C>0$ such that for any $n\in\N$,
\[|\log{Z_n(\bar{\omega})}-\log{Z_n(\bar{\omega}')}|\leq{}Cn^{\delta}.\]
\end{lem}
\begin{proof}
Note first that for any $\omega$,
\begin{equation*}
\tilde{Z}_n(\bar\omega)\ge c_1^ne^{-c_2T_n(\bar\omega)}
\ge c_1^ne^{-c_2d^\alpha n^{1+\alpha\theta}} 
\end{equation*}
since $\bar\omega$ has $\theta$-property. 
This together with Lemma~\ref{greedy}-(ii) allows us to replace $Z_n$ by 
$\tilde{Z}_n$ in the claim.  
For a path $\gamma'$ and $k\in\N$, we define the following sets;
\begin{align*}
 &\Phi_1(\gamma'):=\{\gamma:T_n(\gamma)\leq{}c_2n^{1+2\alpha\theta}, (\gamma,\gamma')\text{ satisfies the condition (i)}\}\\
 &\Phi_{2,k}(\gamma'):=\{\gamma:T_n(\gamma)\leq{}c_2n^{1+2\alpha\theta}, (\gamma,\gamma')\text{ satisfies the condition (ii) with }k\},
\end{align*}
where (i) and (ii) are those in Lemma \ref{6}. 
Then any path $\gamma$ lies in one of the above sets and
\begin{equation*}
\begin{split}
  \tilde{Z}_n(\bar{\omega})
 &\leq{}c^n_1\sum_{\gamma'}\left\{\sum_{\gamma\in\Phi_1(\gamma')}e^{-F_n(\gamma;\bar{\omega})}+\sum_{k\geq{}1}\sum_{\gamma\in\Phi_{2,k}(\gamma')}e^{-F_n(\gamma;\bar{\omega})}\right\}\\
 &\leq{}c^n_1\sum_{\gamma'}e^{-F_n(\gamma';\bar{\omega}')}\left\{\sum_{\gamma\in\Phi_1(\gamma')}e^{c_2n^{(\alpha-1)\zeta+3\theta}}+\sum_{k\geq{}1}\sum_{\gamma\in\Phi_{2,k}(\gamma')}e^{-c_2kn^{\theta}}\right\}\\
 &\leq{}c^n_1\sum_{\gamma'}e^{-F_n(\gamma';\bar{\omega}')}\left\{|\Phi_1(\gamma')|e^{c_2n^{(\alpha-1)\zeta+3\theta}}+\sum_{k\geq{}1}|\Phi_{2,k}(\gamma')|e^{-c_2kn^{\theta}}\right\}.
\end{split}
\end{equation*}
Since for any $\gamma'$, $|\Phi_1(\gamma')|\leq{}n^{2d(M+2\theta)}$ and $|\Phi_{2,k}(\gamma')|\leq{}n^{2d(M+2\theta)k}$ (recall the argument below~\eqref{decomp}), this is further bounded from above by
\begin{eqnarray*}
&\leq{}&c^n_1\sum_{\gamma'}e^{-F_n(\gamma';\bar{\omega}')}\left\{n^{2d(M+2\theta)}e^{c_2n^{(\alpha-1)\zeta+3\theta}}+\sum_{k\geq{}1}n^{2dk(M+2\theta)}e^{-c_2kn^{\theta}}\right\}\\
&\leq{}&c^n_1\sum_{\gamma'}e^{-F_n(\gamma';\bar{\omega}')}\left\{n^{2d(M+2\theta)}e^{c_2n^{(\alpha-1)\zeta+3\theta}}+\sum_{k\geq{}1}e^{-c_2kn^{\theta}/2}\right\}\\
&\leq{}&c^n_1e^{2c_2n^{(\alpha-1)\zeta+3\theta}}\sum_{\gamma'}e^{-F_n(\gamma';\bar{\omega}')}=e^{2c_2n^{(\alpha-1)\zeta+3\theta}}Z_n(\bar{\omega}'). 
\end{eqnarray*}
With the symmetry between $\omega$ and $\omega'$, this implies
\begin{eqnarray*}
|\log{Z_n(\bar{\omega})}-\log{Z_n(\bar{\omega}')}|&\leq&{}\log{\{2e^{2c_2n^{(\alpha-1)\zeta+3\theta}}\}}\\
&\leq&3c_2n^{(\alpha-1)\zeta+3\theta}.
\end{eqnarray*}
If we take $\zeta>0$ sufficiently small so that $3c_2n^{(\alpha-1)\zeta+3\theta}<n^{\delta}$, the proof is completed.
\end{proof}
Thanks to Lemma \ref{5}, we can use Theorem~6.7 in~\cite{BLM13} and we get 
the desired concentration~\eqref{CON} of $\log{Z_n(\bar{\omega})}$.
\section{Non-random fluctuation}\label{sec:fluc}
\subsection{FPP Case}
\hspace{2mm}\\

In this subsection, we deduce the so-called non-random 
fluctuation bound~\eqref{con3;FPP} 
from the concentration bound~\eqref{con1;FPP}. 
This is a well-studied subject in the theory of first passage
percolation and we shall adapt the argument of Zhang 
in~\cite{Zha10} to our setting.
\begin{proof}[Proof of~\eqref{con3;FPP}]
Let $\chi>1/2$, $M=1+\alpha^{-1}$ and $\pi_n=\pi_n^{(n)}$, that is, a 
minimizing path for $T_n(\omega_p)$. 
We define a face to face passage time
\begin{eqnarray}
\begin{split}
\hspace{8mm}\Phi_n(k,l;\omega_p)
=\inf\Biggl\{\sum_{i=k+1}^l |x_{i-1}-x_i|_1^\alpha: 
|x_k|_{\infty}<n^M\textrm{ and }
(i, x_i)\in \omega_p
 \textrm{ for }k<i\le l\Biggr\}
\end{split}
\end{eqnarray}
and introduce the events 
\begin{eqnarray}
\mathcal{A}^{\theta}_1(n)&=&\{\omega_p=\bar{\omega}_p\text{ on }[0,2n]\times[-n^M,n^{M}]^d\},\\
\mathcal{A}_2(n)&=&\left\{\begin{array}{c}
\textrm{There exists a minimizing path for }\\
\Phi_n(n,2n;\omega_p)\textrm{ starting at $(n,x)$ with }|x|_{\infty}\le 1/2
\end{array}\right\}. 
\end{eqnarray}
\begin{lem}
\label{good-events}
We fix a sufficiently small constant $\theta>0$. Then, for all sufficiently large $n$, 
\begin{enumerate}
\item $Q(\mathcal{A}^{\theta}_1(n))\ge 1-\exp\{-C_4n^{\theta}\}$;
\item $Q(\mathcal{A}_2(n))\ge 2^{-d}n^{-dM}$.
\end{enumerate} 
\end{lem}
\begin{proof}
The first assertion is a direct consequence of Lemma \ref{approx}. Translation invariance implies that the probabilities
\begin{eqnarray}
\begin{split}
 &Q(\textrm{There exists a minimizing path for }
 \Phi_n(0,n;\omega_p)\\
 &\qquad \textrm{ starting at $(0,x)$ with }
 |x-y|_{\infty}\le 1/2) 
\end{split}
\end{eqnarray} 
for $y\in \Z^d\cap (-n^M, n^M)^d$ are the same as $Q(\mathcal{A}_2(n))$. 
Since the union of the above events has probability one and the
number of possible $y$'s are less than $(2n^M)^d$, we 
are done. 
\end{proof}

With this lemma and~\eqref{con1;FPP}, 
we can complete the proof of~\eqref{con3;FPP}. 
Note first that on the event $\mathcal{A}_1(n)$, we have
$\pi_{2n}(n)\in (-n^M, n^M)^d$ as the displacement of 
$\pi_{2n}$ until time $n$ is at most $2n\cdot{}n^{\theta}\vee{}(2n)^{\theta+\alpha^{-1}}<n^M$ for any sufficiently small $\theta$.
As a result
\begin{eqnarray}
T_{2n}(\omega_p)\ge T_n(\omega_p)+\Phi_n(n,2n;\omega_p)
\end{eqnarray}
since the second half of $\pi_{2n}$ is candidates of 
the face to face minimizing paths. 
On the other hand, from Lemma \ref{2}, for any $\zeta>0$, if we take $\theta>0$ sufficiently small, 
\begin{eqnarray}
\Phi_n(n,2n;\omega_p)
\ge T(n,2n;\omega_p)-2d^{\alpha}
\vee Cn^{(\alpha-1)\zeta+\theta}
\end{eqnarray}
with some $C>0$ on $\mathcal{A}=\bigcap_{1\le i\le 2}\mathcal{A}_i(n)$ 
since only possible differences come from the starting points,
which can be controlled by using the mean value theorem.
Therefore on $\mathcal{A}$, for $\chi>1/2$ and sufficiently large $n\in\N$, 
we have the following almost super-additivity: 
\begin{eqnarray}
T_{2n}(\omega_p)\ge T_n(\omega_p)+T(n,2n;\omega_p)
-n^{\chi}.
\label{super-add}
\end{eqnarray}
Now we use~\eqref{con1;FPP} to obtain
\begin{eqnarray}
Q\left(|T_n(\omega_p)-Q[T_n(\omega_p)]|
>n^{\chi}\right)
\le C_1\exp\set{-C_2 n^{\lambda}}
\end{eqnarray}
and the same bound for $T_{2n}(\omega_p)$ and 
$T_{2n}(\omega_p)$. 
These bounds and Lemma~\ref{good-events} show that
for all sufficiently large $n$,
\begin{eqnarray}
\mathcal{A}\cap\bigcap_{(k,l)\in\{(0,2n), (0,n), (n,2n)\}}
\set{|T(k,l;\omega_p)-Q[T(k,l;\omega_p)]|
\le n^{\chi}}
\end{eqnarray}
has positive probability and in particular non-empty. 
Hence we can replace the passage times in~\eqref{super-add} by 
their expectation at the cost of extra $-3n^{\chi}$
on the right-hand side to obtain
\begin{eqnarray}
\frac{1}{2n}Q[T_{2n}(\omega_p)]
\ge \frac{1}{n}Q[T_n(\omega_p)]-4n^{\chi-1}.
\end{eqnarray}
Iterating this, we arrive at
\begin{eqnarray}
\frac{1}{n}Q[T_n(\omega_p)] 
\le \frac{1}{2^kn} Q[T_{2^kn}(\omega_p)]
+4n^{\chi-1}\sum_{j=1}^{k-1}2^{(\chi-1)j}.
\end{eqnarray}
and letting $k\to\infty$, 
$Q[T_n(\omega_p)] \le n\mu_p+Cn^{\chi}$ follows. On the other hand, since it is easy to check the usual subadditivity, we have \eqref{con3;FPP}. From \eqref{con1;FPP} and \eqref{con3;FPP}, we have \eqref{con2;FPP}.
\end{proof}
\hspace{2mm}\\
\subsection{Free energy Case}
\hspace{2mm}\\

The proof is almost the same as that for FPP, so we only mention the outline here.
To prove Theorem \ref{21}, we again use the argument of Zhang. Let
\[Z(k,\ell,x,\omega):=c^n_1\sum_{\gamma:\gamma(k)=x}e^{\beta\sum^{\ell}_{i=k+1}\eta(i,\gamma(i))}e^{-c_2\sum^{\ell}_{i=k+1}\Delta\gamma(i)^{\alpha}}\]
Let $M=1+{1 \over\alpha}$ and we define an event $\mathcal{A}$ by
\[\mathcal{A}:=\{Z(n,2n,0,\omega)=\sup_{x\in[-n^M/2,n^M/2]^d}{Z(n,2n,x,\omega)},Z_{2n}(\omega)\leq{}2\tilde{Z}_{2n}(\omega)\}.\]
One can show, as in the case of FPP, that for sufficiently large $n$,
\[Q(\mathcal{A})\geq{}\frac{1}{2n^{Md}}.\]
 On this event,
\begin{equation*}
\begin{split}
  \log{Z_{2n}}-\log 2&\leq{}\log{\tilde{Z}_{2n}}\\
&\leq
 \log{\left\{c^{2n}_1\sum_{\underset{\gamma(0)=0}{\gamma:\max_{1\le i\le n}|\gamma(i)|_\infty\le n^M/2}}e^{-F_{2n}(\gamma;\omega)}\right\}}\\
 &\leq\log{\left\{c^n_1\sum_{x\in[-n^M/2,n^M/2]^d}\sum_{\underset{\gamma(0)=0}{\gamma:\gamma(n)=x}}e^{-F_n(\gamma;\omega)}Z(n,2n,x,\omega)\right\}},
\end{split}
\end{equation*}
where we have used the argument below~\eqref{decomp} in the second inequality and divided paths at its $n$-th point in the third ineqality. Thanks to the choice of $\mathcal{A}$, this is further bounded from above by
\begin{equation*}
\begin{split}
 &\log{\left\{c^n_1\sum_{x\in[-n^M/2,n^M/2]^d}\sum_{\underset{\gamma(0)=0}{\gamma:\gamma(n)=x}}e^{-F_n(\gamma;\omega)}Z(n,2n,0,\omega)\right\}}\\
 &\leq\log{(Z_n(\omega)Z(n,2n,0,\omega))}=
 \log{Z_n(\omega)}+\log{Z(n,2n,0,\omega)}.
\end{split}
\end{equation*}
From the concentration \eqref{con1;FE} and this, for any $\chi>1/2$, there exist constant $C>0$ independent of $n$ such that,
\[Q[\log{Z_{2n}}]\leq{}2Q[\log{Z_n}]+Cn^{\chi}
\]
In a similar way to the proof of Theorem \ref{21}, we have for any $k$,
\[\frac{Q[\log{Z_{2^kn}}]}{2^kn}\leq{}\frac{Q[\log{Z_{n}}]}{n}+C'n^{\chi-1},
\]
where $C'>0$ is a constant independent of $n$.
As $k\to\infty$, we have
\[\varphi(p,\beta)\leq{}\frac{1}{n}Q[\log{Z_{n}}]+C'n^{\chi-1},
.\]
 Finally, we shall derive the converse estimate by a similar way. Indeed, we replace $\mathcal{A}$ by $\{Z(n,2n,0,\omega)=\min_{x\in[-n^M/2,n^M/2]^d}{Z(n,2n,x,\omega)},Z_{n}(\omega)\leq{}2\tilde{Z}_{n}(\omega)\}$, and on this event, we have
\begin{equation*}
\begin{split}
\log{Z_{2n}}\ge{}\log{\tilde{Z}_n}+\log{Z(n,2n,0,\omega)}\ge \log{2}+\log{Z_n}+\log{Z(n,2n,0,\omega)},
\end{split}
\end{equation*}
 which implies
\[\varphi(p,\beta)\geq{}\frac{1}{n}Q[\log{Z_{n}}]-C'n^{\chi-1}.
\]
This implies \eqref{con11;FE}. From \eqref{con11;FE} and \eqref{con1;FE}, it is immediate to prove \eqref{con2;FE}.\\
\hspace{2mm}\\
\section{Proof of continuity results}\label{sec:cont}
Thanks to the non-random fluctuation results, it suffices to show the continuity of $Q[T_n]$ and $Q[\log{Z_n}]$ for fixed $n\in\N$.
\begin{lem}\label{finite-vol}
For any $n\in\N$ fixed, $Q[T_n(\omega_p)]$ is continuous function of $p\in[0,1]$.
\end{lem}
\begin{proof}
Note that $T_n(\cdot)$ is continuous on the set of 
locally finite point configurations with respect to the 
vague topology. 
Indeed for any $\omega$, the 
definition of passage time tells us that only points inside a 
compact set matter. By local finiteness we can choose the compact 
set in such a way that its boundary 
contains no points of $\omega$. Now if $\omega_N\to\omega$ vaguely
as $N\to\infty$, then the points of $\omega_N$ 
inside the compact set converge to those of $\omega$ in the 
Hausdorff metric and then it easily follows that 
$T_n(\omega_N)\to T_n(\omega)$. 
By using the Skorohod representation theorem, we may assume that 
$Q$-almost surely, $\omega_p\to \omega_p'$ vaguely as $p\uparrow p'$. 
Then by continuity, $T_n(\omega_p)\to T_n(\omega_p')$
as $p\uparrow p'$, $Q$-almost surely. From this, $T_n(\omega_p)$ is continuous in $p\in[0,1]$. Since the uniform integrability of 
$\{T_n(\omega_p)\}_{p\in (0,1]}$ follows from Lemma~\ref{greedy}, $L^1(Q)$ convergence follows. 
\end{proof}
The following lemma can be proved by a similar way to Lemma \ref{finite-vol}.
\begin{lem}\label{iii}
  $Q[\log{Z_n}]$ is jointly continuous on $[0,1)\times[-\infty,\infty)$.
\end{lem}
From Lemmas \ref{finite-vol} and \ref{iii}, we obtain Corollaries \ref{conti1} and \ref{conccc;FE}.
\section*{Acknowledgements}
We gratefully acknowledge useful conversations with Ryoki Fukushima.

\end{document}